\documentclass[12pt]{article}

\usepackage{amsmath,amssymb,amsthm,amsfonts}
\usepackage{fullpage}
\usepackage{enumerate}

\newtheorem{theorem}{Theorem}

\newtheorem{lemma}[theorem]{Lemma}

\newtheorem{conjecture}[theorem]{Conjecture}
\newtheorem{defin}[theorem]{Definition}

\newtheorem{examp}[theorem]{Example}

\newtheorem{rema}[theorem]{Remark}

\begin{document}

\title{Unbounded discrepancy in Frobenius numbers}

\author{Jeffrey Shallit\\
School of Computer Science \\
University of Waterloo\\
Waterloo, ON  N2L 3G1 \\
Canada\\
{\tt shallit@cs.uwaterloo.ca }\\
\and
James Stankewicz\footnote{Partially supported by NSF VIGRE grant
DMS-0738586}\\
Department of Mathematics \\
University of Georgia \\
Athens, GA 30602 \\
USA \\
{\tt stankewicz@gmail.com }}

\maketitle

\begin{abstract}

Let $g_j$ denote the largest integer that is represented exactly $j$
times as a non-negative integer linear combination of $\lbrace x_1,
\ldots, x_n\rbrace$.  We show that for any $k > 0$, and
$n = 5$, the quantity $g_0 - g_k$ is unbounded. Furthermore, we
provide examples with $g_0 > g_k$ for $n \geq 6$ and $g_0 > g_1$
for $n \geq 4$.
\end{abstract}

\section{Introduction}

     Let $X = \lbrace x_1, x_2, \ldots, x_n \rbrace$ be a set of distinct positive
integers such that $\gcd(x_1, x_2, \ldots, x_n) = 1$.  The
{\sl Frobenius number} $g(x_1, x_2, \ldots, x_n)$ is defined to be the
largest integer that cannot be expressed as a non-negative integer
linear combination of the
elements of $X$.   For example, $g(6, 9, 20) = 43$.  

       The Frobenius number --- the name comes from the fact that Frobenius
mentioned it in his lectures, although he apparently never wrote about it ---
is the subject of a huge literature, which is admirably summarized in the
book of Ram\'{\i}rez Alfons\'{\i}n \cite{ramirez}.  

     Recently, Brown et al.\ \cite{brown} considered a generalization of
the Frobenius number, defined as follows:  $g_j(x_1, x_2, \ldots, x_n)$ is
largest integer 
having exactly $j$ representations as a non-negative
integer linear combination of $x_1, x_2, \ldots, x_n$. 
(If no such integer exists, Brown et al.\ defined $g_j$ to be $0$, but
for our purposes, it seems more reasonable to leave it undefined.)
Thus $g_0$ is just
$g$, the ordinary Frobenius number.     They observed that, for a fixed
$n$-tuple $(x_1, x_2, \ldots, x_n)$, the function $g_j(x_1, x_2, \ldots, x_n)$
need not be increasing (considered as a function of $j$).  For example,
they gave the example
$g_{35} (4,7,19) = 181$ while $g_{36}(4,7,19) = 180$.  They asked if there
are examples for which $g_1 < g_0$.  Although they did not say so,
it makes sense to impose the condition that 

\smallskip

no $x_i$ can be written as a non-negative integer linear
combination of the others, (*)

\smallskip

\noindent for otherwise we have trivial examples such as 
$g_0(4, 5, 8, 10) = 11$ and $g_1 (4,5,8,10) = 9$.     We call a tuple
satisfying (*) a {\sl reasonable} tuple.

     In this note we show that the answer to the question of
Brown et al.\ is ``yes'', even for reasonable tuples.
For example, it is easy to verify that
$g_0 (8,9,11,14,15) = 21$, while
$g_1 (8,9,11,14,15) = 20$.  But we prove much more:
we show that 
$$g_0 (2n-2, 2n-1, 2n, 3n-3, 3n) = n^2 - O(n),$$ while for
any fixed $k \geq 1$ we have $g_k (2n-2, 2n-1, 2n, 3n-3, 3n) = O(n)$.
It follows that for this parameterized
$5$-tuple and all $k \geq 1$,
we have $g_0 - g_k \rightarrow \infty$ as $n \rightarrow \infty$.

\section{The main result}

      We define $X_n = \lbrace 2n-2, 2n-1, 2n, 3n-3, 3n \rbrace$.
It is easy to see that this is a reasonable
$5$-tuple for $n \geq 5$.
If we can write $t$ as a non-negative linear combination
of the elements of $X_n$, we say $t$ has a representation or is
representable.

We define $R(j) $ to be the number of distinct representations of $j$
as a non-negative integer linear combination of the elements of $X_n$.

\begin{theorem}
\begin{enumerate}[$($a$)$]
\item $g_k (X_n) = (6k+3)n - 1$ for $n > 6k+3$, $k \geq 1$.
\item $g_0(X_n) = n^2 - 3n +1$ for $n \geq 6$;
\end{enumerate}
\label{ref1}
\end{theorem}

Before we prove Theorem~\ref{ref1}, we need some lemmas.

\begin{lemma}
\begin{enumerate}[$($a$)$]
\item $R( (6k+3)n - 1) \geq k$ for $n \geq 4$ and $k \geq 1$.
\item $R( (6k+3)n - 1) = k$ for $n > 6k+3$ and $k \geq 1$.
\end{enumerate}
\label{lem1}
\end{lemma}

\begin{proof}
First, we note that 
\begin{equation}
(6k+3)n - 1 = 1 \cdot (2n-1) + (3t-1) \cdot (2n) + (2(k-t)+1)\cdot (3n)
\label{eq1}
\end{equation}
for any integer $t$ with $1 \leq t \leq k$.  This provides at least $k$ distinct
representations for $(6k+3)n - 1$ and proves (a).  We call these
$k$ representations {\sl special}.

To prove (b), we need to see that the $k$ special representations given by
(\ref{eq1}) are, in fact, all representations that can occur.  

Suppose that $(a,b,c,d,e)$ is a $5$-tuple of non-negative integers such that
\begin{equation}
a(2n-2) + b(2n-1) + c(2n) + d(3n-3) + e(3n)  = (6k+3) n - 1 .
\label{eq2}
\end{equation}
Reducing  this equation modulo $n$, we get
$ -2a -b -3d \equiv -1 $ (mod $n$).  Hence there exists an integer
$m$ such that $2a + b + 3d  = mn + 1$.  Clearly $m$ is non-negative.  
There are two cases to consider:  $m = 0$ and $m \geq 1$.

If $m = 0$, then $2a + b + 3d = 1$, which, by the non-negativity of the
coefficients $a, b, d$ implies that $a=d=0$ and $b = 1$.  Thus by
(\ref{eq2}) we get $ 2n-1 + 2cn + 3en  = (6k+3) n - 1$, or
\begin{equation}
2c + 3e = 6k+1.
\label{eq3}
\end{equation}
Taking both sides modulo $2$, we see that 
$e \equiv 1$ (mod $2$), while taking both sides modulo $3$, we see that
$c \equiv 2$ (mod $3$).  Thus we can write $e = 2r+1$, $c = 3s-1$, and
substitute in (\ref{eq3}) to get $k = r + s $.   Since $s \geq 1$, it follows that
$0 \leq r \leq k-1$, and this gives our set of $k$ special representations
in (\ref{eq1}).

If $m \geq 1$, then $n +1 \leq mn+1 = 2a+b + 3d$, so
$n \leq 2a+b+3d-1$.  However, we know that
$(6k+3)n - 1 \geq a(2n-2) + b(2n-1) + d(3n-3) > (n-1) (2a+b+3d)$.  
Hence $(6k+3)n > (n-1)(2a+b+3d) + 1 > (n-1)(2a+b+3d-1) \geq (n-1)n$.
Thus $6k+3 > n-1$.    It follows that if $n > 6k+3$, then this case
cannot occur, so all the representations of $(6k+3)n - 1$ are 
accounted for by the $k$ special representations given in (\ref{eq1}).
\end{proof}

We are now ready to prove Theorem~\ref{ref1} (a).

\begin{proof}
We already know from Lemma~\ref{lem1} that for $n > 6k+3$, the
number $N := (6k+3)n - 1$ has exactly $k$ representations.  It now suffices
to show that if $t$ has exactly $k$ representations, for $k \geq 1$, then 
$t \leq N$.

We do this by assuming $t$ has at least one representation, say
$t = a(2n-2) + b(2n-1) + c(2n) + d(3n-3) + e(3n)$, for
some $5$-tuple of non-negative integers $(a,b,c,d,e)$.
Assuming these integers are large enough (it suffices to assume 
$a,b,c,d,e \geq 3$),
we may take advantage of the internal 
symmetries of $X_n$ to obtain additional representations with
the following swaps.

\begin{enumerate}[$($a$)$]

\item $3(2n) = 2(3n)$;
hence 
$$a(2n-2) + b(2n-1) + c(2n) + d(3n-3) + e(3n) $$
$$ = a(2n-2) + b(2n-1) + (c+3)(2n) + d(3n-3) + (e-2)(3n).$$ 

\item $3(2n-2) = 2(3n-3)$; hence $$ a(2n-2) + b(2n-1) + c(2n) + d(3n-3) + e(3n)$$ $$= (a+3)(2n-2) + b(2n-1) + c(2n) + (d-2)(3n-3) + e(3n).$$

\item $2n-2 + 2n = 2(2n-1)$; hence $$ a(2n-2) + b(2n-1) + c(2n) + d(3n-3) + e(3n)$$ $$= (a+1)(2n-2) + (b-2)(2n-1) + (c+1)(2n) + d(3n-3) + e(3n).$$

\item $2n-2 + 2n-1 + 2n = 3n-3 + 3n$; hence $$ a(2n-2) + b(2n-1) + c(2n) + d(3n-3) + e(3n)$$ $$= (a+1)(2n-2) + (b+1)(2n-1) + (c+1)(2n) + (d-1)(3n-3) + (e-1)(3n).$$

\end{enumerate}

We now do two things for each possible swap:  first, we show that the requirement
 that $t$ have exactly $k$ representations imposes upper bounds on the 
size of the coefficients.
Second, we swap until we have a representation which can be conveniently bounded in terms of $k$.

\begin{enumerate}[$($a$)$]

\item  If $\lfloor {e \over 2} \rfloor +
\lfloor {c \over 3} \rfloor \geq k$, we can find at least $k+1$
representations of $t$.  Thus we can find a representation of $t$ with
$c \leq 2$ and $e \leq 2k-1$.

\item   Similarly, if $\lfloor {d \over 2} \rfloor +
\lfloor {a \over 3} \rfloor  \geq k$, we can find at least $k+1$
representations of $t$.  Thus we can find a representation of $t$
with $d \leq 2k-1$ and $a \leq 2$. Combining this with (a), we can
 find a representation with $a,c \leq 2$ and $d + e \leq 2k -1$.

\item If $\lfloor {b \over 2} \rfloor + \min\{ a,c \} \geq k$, 
we can find at least $k+1$ representations of $t$.   
 Thus we can find a representation of $t$ with $| b - \min\{ a,c \} | \leq 1$.  
If we start with the assumption $a, c \leq 2$, this ensures that
$\min\{a,b,c\} \leq \lfloor {{a+b+c}\over 3} \rfloor \leq
\min\{ a,b,c \} + 1$ and $\max\{a,b,c\} - \min\{a,b,c\} \leq 3$.  

\item If $\min\{ a,b,c \} + \min\{ d,e \} \geq k$ we can find at least $k+1$
representations of $t$.  When this swap is followed by (a) or (b) (if necessary)
we can find a representation with $d + e \leq 2k -1$, $a+b+c \leq 3$ and $a,c \leq 2$.

\end{enumerate}

Putting this all together, we see that $t \leq (2n-1) + 2 (2n) +
(2k-1)(3n) = (6k+3)n - 1$, as desired.
\end{proof}

In order to prove Theorem~\ref{ref1} (b), we need a lemma.

\begin{lemma}
The integers $k(n-1), k(n-1)+1, \ldots, kn$ are representable for
$k = 2$ and $k \geq 4$ and for $n \geq 4$.  
\label{lem2}
\end{lemma}

\begin{proof}
We prove the result by induction on $k$.    The base cases 
are $k = 2,4$, and we have the representations given below:

\begin{eqnarray*}
4n-4 &=& 2 (2n-2) \\
4n-3 &=& (2n-2) + (2n-1) \\
4n-2 &=& 2(2n-1) \\
4n-1 &=& (2n-1) + (2n) \\
4n &=& 2 (2n).
\end{eqnarray*}

Now suppose $ln-m$ is representable for $4 \leq l < k$ and
$0 \leq m \leq l$.   We want to show that $kn-t$ is representable for
$0 \leq t \leq k$.    There are three cases, depending on $k$ (mod $3$).

If $k \equiv 0$ (mod 3), and $k \geq 4$, then $(k-2)n - t = kn - t - 2n$ is
representable if $t \leq k-2$; otherwise $(k-2)n - t + 2 = kn - t - (2n-2)$
is representable.   By adding $2n$ or $2n+2$, respectively, we get a 
representation for $kn-t$.

If $k \equiv 1$ (mod 3), and $k \geq 4$, or if 
$k \equiv 2$ (mod 3), then $(k-3)n - t = kn -t - 3n$ is
representable if $t \leq k-3$; otherwise $(k-3)n -t + 3 = kn-t - (3n-3)$
is representable.  By adding $3n$ or $3n+3$, respectively, we get a
representation for $kn-t$.  
\end{proof}

Now we prove Theorem~\ref{ref1} (b).

\begin{proof}
First, let's show that every
integer $> n^2 - 3n+1$ is representable.  Since if $t$ has a representation,
so does $t+2n-2$, it suffices to show that the $2n-2$ numbers
$n^2 -3n+2, n^2-3n+3, \ldots, n^2-n-1$ are representable.  

We use Lemma~\ref{lem2} with $k = n-2$ to see that the numbers
$(n-2)(n-1) = n^2-3n+2, \ldots, (n-2)n = n^2-2n$ are all representable.  
Now use Lemma~\ref{lem2} again with $k = n-1$ to see that the numbers
$(n-1)(n-1) = n^2-2n+1, \ldots, (n-1)n = n^2-n$  are all representable.
We therefore conclude that every integer $> n^2-3n+1$ has a representation.

Finally, we show that $n^2-3n+1$ does not have a representation.
Suppose, to get a contradiction, that it does:
$$ n^2 -3n+1 = a(2n-2) + b(2n-1) + c(2n) + d(3n-3) + e(3n).$$
Reducing modulo $n$ gives
$1 \equiv -2a -b -3d $ (mod $n$), so there exists an integer $m$ such
that $2a+b+3d = mn-1$.  Since $a, b, d$ are non-negative, we must have
$m \geq 1$.  

Now $n^2-3n+1 \geq a(2n-2)+ b(2n-1) + d(3n-3) > (n-1)(2a+b+3d)$.  
Thus 
\begin{equation}
n^2 -3n+1 \geq (n-1)(mn-1) = mn^2 -(m+1)n + 1.
\label{eq5}
\end{equation}
If $m = 1$,
we get $n^2 - 3n+1 \geq n^2 - 2n + 1$, a contradiction.  Hence $m \geq 2$.
 From (\ref{eq5}) we get $(m-1)n^2 - (m-2)n \leq 0$.  Since $n \geq 1$, we get
$(m-1)n - (m-2) \leq 0$, a contradiction.
\end{proof}

\section{Additional remarks}

One might object to our examples because the numbers are not pairwise
relatively prime.  But there also exist reasonable
$5$-tuples with $g_0 > g_1$ for which
all pairs are relatively prime:  for example,
$g_0(9,10,11,13,17) = 25$, but $g_1(9,10,11,13,17) = 24$.
More generally one can use the techniques in this paper to show that
$g_0(10n-1, 15n-1, 20n-1, 25n, 30n-1) = 50n^2 -1$ and
$g_1(10n-1, 15n-1, 20n-1, 25n, 30n-1) = 50n^2 - 5n$ for $n \geq 1$, so that
$g_0 - g_1 \rightarrow \infty$ as $n \rightarrow \infty$.

For $k \geq 2$, let $f(k)$ be the least non-negative integer $i$
such that
there exists a reasonable $k$-tuple $X$ with $g_i(X) > g_{i+1}(X)$.  
A priori $f(k)$ may not exist.  For example, if $k = 2$, then we have
$g_i (x_1, x_2) = (i+1)x_1x_2 - x_1 - x_2$, so 
$g_i (x_1, x_2) < g_{i+1} (x_1, x_2)$ for all $i$.
Thus $f(2)$ does not exist.
In this paper, we have shown that $f(5) = 0$.   

This raises the obvious question of other values of $f$.

\begin{theorem}
We have $f(i) = 0$ for $i \geq 4$.
\end{theorem}

\begin{proof}
As mentioned in the Introduction, the example $(8,9,11,14,15)$ shows that
$f(5) = 0$.

For $i = 4$, we have the example $g_0(10,15,32,48) = 101$ and
$g_1(10,15,32,48) = 99$, so $f(4) = 0$.  (This is the reasonable quadruple 
with $g_0 > g_1$ that minimizes the largest element.)

We now provide a class of examples for $i \geq 6$.  For $n \geq 6$
define $X_n$ as follows:
$$X_n = (n+1, n+4, n+5, [n+7..2n+1], 2n+3, 2n+4),$$
where by $[a..b]$ we mean the list $a, a+1, a+2, \ldots, b$.

For example, $X_8 = (9, 12, 13, 15, 16, 17, 19, 20)$.  Note that $X_n$ is
of cardinality $n$.  We make the following three claims for $n \geq 6$.

\begin{enumerate}[$($a$)$]
\item $X_n$ is reasonable.
\item $g_0 (X_n) =  2n+7$.
\item $g_1 (X_n) =  2n+6$.
\end{enumerate}

(a):  To see that $X_n$ is reasonable, assume that some element $x$ is in the
${\mathbb N}$-span of the other elements.  Then either $x = ky$ for some
$k \geq 2$, where $y$ is the smallest element of $X_n$, or $x \geq y+z$, where
$y, z$ are the two smallest elements of $X_n$.   It is easy to see both
of these lead to contradictions.

(b) and (c): Clearly $2n+7$ is not representable, and  $2n+6$ has the
single representation $(n+1) + (n+5)$.  It now suffices to show that every
integer $\geq 2n+8$ has at least two representations.  And to show this,
it suffices to show that all integers in the range $[2n+8..3n+8]$ have
at least two representations.  

Choosing $(n+4) + [n+7..2n+1]$ and $(n+5)+[n+7..2n+1]$ gives two distinct
representations for all numbers in the interval $[2n+12..3n+5]$.  So it
suffices to handle the remaining cases $2n+8, 2n+9, 2n+10, 2n+11,
3n+6, 3n+7, 3n+8$.  This is done as follows:

\begin{alignat*}{2}
2n+8 &= (n+1)+(n+7) & \ = & \ 2(n+4) \\
2n+9 &= (n+4)+(n+5) & \ =  & \
	\begin{cases}
		3(n+1), &  \text{if $n = 6$}; \\
		(n+1)+(n+8), & \text{if $n \geq 7$.}
	\end{cases} \\
2n+10 &= 2(n+5) & \ = & \ 
	\begin{cases}
		(n+1)+(2n+3), & \text{if $n = 6$}; \\
		3(n+1), & \text{if $n = 7$}; \\
		(n+1)+(n+9), & \text{if $n \geq 8$.}
	\end{cases} \\
2n+11 &= (n+4)+(n+7) & \ = & \
	\begin{cases}
		(n+1)+(2n+4), & \text{if $n = 6$}; \\
		(n+1)+(2n+3), & \text{if $n = 7$}; \\
		3(n+1) , & \text{if $n = 8$}; \\
		(n+1)+(n+10), & \text{if $n \geq 9$.}
	\end{cases} \\
3n+6 &=  2(n+1) + (n+4) & \ = & \ (n+5) + (2n+1) \\
3n+7 &=  2(n+1) + (n+5) & \ = & \ (n+4) + (2n+3) \\
3n+8 &=  (n+5) + (2n+3) & \ = & \ (n+4) + (2n+4).
\end{alignat*}
			
\end{proof}

We do not know the value of $f(3)$.
The example 
\begin{align*}
g_{14}(8,9,15) &= 172 \\
g_{15}(8,9,15) &= 169 
\end{align*}
shows that $f(3) \leq 14$.

\begin{conjecture} $f(3) = 14$.
\end{conjecture}

We have checked all triples with largest element $\leq 200$, but have
not found any counterexamples.

\section{Acknowledgments}

We thank the referee for useful comments.
Thanks also go to Dino Lorenzini who sent us a list of comments after
we submitted this paper. Among them was an encouragement to make more
use of the formula of Brown et al, which led to the example that shows that
$f(4) = 0$.


\begin{thebibliography}{9}

\bibitem{brown}
A. Brown, E. Dannenberg, J. Fox, J. Hanna, 
K. Keck, A. Moore, Z. Robbins, B. Samples, 
and J. Stankewicz, On a generalization of the Frobenius
number, {\it J. Integer Sequences} {\bf 13} (2010), Article 10.1.4.

\bibitem{ramirez}
J. L. Ram\'{\i}rez Alfons\'{\i}n.
\newblock {\it The Diophantine Frobenius Problem}.
\newblock Oxford University Press, 2005.


\end{thebibliography}
\end{document}